\documentclass[journal]{IEEEtran}

\usepackage{cite}

\ifCLASSINFOpdf
\usepackage[pdftex]{graphicx}
\graphicspath{{../pdf/}{../jpeg/}}
\DeclareGraphicsExtensions{.pdf,.jpeg,.png}
\else
\fi

\usepackage[cmex10]{amsmath}

\usepackage{amssymb, bm}

\usepackage{comment}
\usepackage{array,multirow}
\usepackage{tikz}

\usepackage{accents}
\newcommand{\obar}[1]{\overline{#1}}
\newcommand{\ubar}[1]{\underline{#1}}
\usepackage{bbm}
\usepackage{dsfont}
\usepackage{bbold}

\usepackage{algorithm}
\usepackage{algorithmic}

\usepackage{amsthm}
\usepackage{enumerate}

\theoremstyle{definition}
\newtheorem{definition}{Definition}
\newtheorem{assumption}{Assumption}
\newtheorem{theorem}{Theorem}
\newtheorem{corollary}{Corollary}
\newtheorem{lemma}{Lemma}

\newtheorem{remark}{Remark}
\usepackage{dsfont}

\usepackage{mathtools}

\usepackage{color,soul}

\begin{document}
\title{Convex Restriction of Power Flow Feasible Sets}
\author{Dongchan Lee, Hung D. Nguyen, Krishnamurthy Dvijotham, and Konstantin~Turitsyn
\thanks{
This work was supported by fundings from the U.S. Department of Energy Office of Electricity as part of the DOE Grid Modernization Initiative and the NSF EPCN award 1809314.}}
\maketitle

\begin{abstract}
The convex restriction of the power flow feasible sets identifies the convex subset of power injections where the solution for power flow is guaranteed to exist and satisfy the operational constraints. In contrast to convex relaxations, the convex restriction provides a sufficient condition for power flow feasibility and is particularly useful for problems involving uncertainty in the power generation and demand. In this paper, we present a general framework of constructing convex restriction of an algebraic set defined by equality and inequality constraints and apply the framework to power flow feasibility problem. The procedure results in convex quadratic constraints that provide a sufficiently large region for practical operation of the grid.
\end{abstract}


\IEEEpeerreviewmaketitle

\section{Introduction}
{\color{black} Power flow equations are at the core of steady-state analysis of the power grid \cite{cutsem98, kundur94}. Optimal Power Flow (OPF), State estimation and security assessment rely on AC power flow equations to model the grid. Power flow equations determine internal states of the system such as voltage magnitude and phase angles given the profile of generation and loads. The nonlinearity of AC power flow equations creates computational bottlenecks and challenges in those studies.

In state estimations and security assessments, the state variables are determined using numerical algorithms such as the Newton-Raphson method or Backward-Forward sweep method. The disadvantage of using numerical algorithms is that they require a deterministic operating point to find the exact state solution. When uncertainties in the generation and load profile are introduced, there is no easy way to tell whether there will be a state solution satisfying the AC power flow equations without running iterative algorithms.

In OPF problems, power flow equations enter as nonlinear equality constraints and result in a non-convex optimization problem, which is NP-hard \cite{lesieutre05} even for radial networks \cite{lehmann16,low14_1}. Convex relaxations of power flow equation have been studied extensively for solving OPF problems \cite{low14_1,lavaei12, coffrin16}. The convex relaxation provides an outer-approximation of the feasibility set, and it is a necessary condition to satisfy the power flow equations. Solving the optimization problem over the relaxed set provides a lower bound on the optimal cost, but the resulting solution may not be feasible and risks the system security \cite{cui17}. Moreover, it provides limited insights and characterization of the feasibility set because the non-convex boundaries inside the feasibility set disappear in convex relaxations \cite{molzahn17}.

This paper is concerned with finding the inner approximation of the feasibility set. The \textit{convex restriction} is a convex subset of the feasibility set, which provides a sufficient condition for satisfying power flow equations with operational constraints. Figure \ref{fig_inner_approx} shows the comparison between the convex relaxation and restriction. The benefit of studying inner approximation is that the security of the system is guaranteed, which is the top priority in the operation of power grids. Moreover, it provides a region where the system is safe to operate, and this region can be used as a metric for robustness against uncertain power injections from renewables and loads. While there are many potential applications of convex restriction, deriving a tractable sufficient condition for the feasibility of power flow equations has remained as a challenge.

\begin{figure}[b]
	\centering
	\includegraphics[width=2.2in]{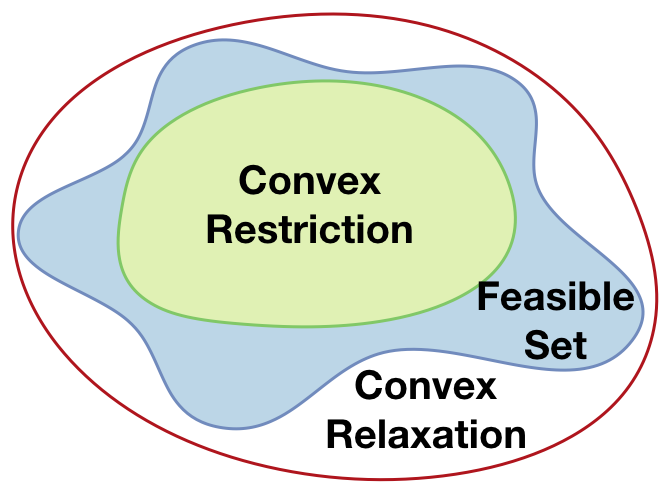}
	\caption{Illustration of the convex restriction and convex relaxation of a non-convex feasibility set}
	\label{fig_inner_approx}
\end{figure}

The search for tractable convex restriction of power flow equations started in \cite{wu82} to find the security region where the system is safe to operate. In recent years, a number of effort has been made in providing the inner approximation of the feasibility set, but there have been severe limitations in terms of its applicability to practical systems. Most of the progress were made with certain modeling assumption such as radial structure \cite{bolognani16,yu15,wang16,wang17,dvijotham17}, lossless network \cite{simpson17_1}, and decoupled power flow model \cite{simpson16}. Recent efforts made significant progress with general meshed networks, but they still suffer from scalability and conservativeness \cite{dvijotham15,dvijotham18,chiang18}. In \cite{nguyen17}, the inner approximation with Brouwer's fixed point theorem showed promising results for general power grid models. One of the limitations of this approach was that it required solving a non-convex optimization problem to construct the convex restriction. In this paper, we alleviate this limitation by describing the set in a lifted space and give a closed form expression.  

We propose an analytical procedure to construct convex restrictions of AC power flow equations with operational constraints. Our technique relies on envelopes over the nonlinearity involved in power flow equations, and these envelopes for common nonlinear functions such as bilinear and trigonometric functions will be provided. Our envelopes show an interesting relation to the QC relaxation for OPF, which employs convex envelopes to contain the nonlinearities \cite{coffrin16}. It will be later shown that the envelopes for restriction have dual features compared to the ones used in relaxations. Moreover, the construction relies on bounds over intervals, which has been studied in interval analysis and uncertainty propagation techniques \cite{moore09,jaulin01}. The interval analysis also deals with finding the inner approximation of sets described by constraints, but the work has been limited to a subclass of problems such as linear equations or decentralized nonlinear equations \cite{goldsztejn05,goldsztejn06,olivier13}. To the best of author's knowledge, there is no tractable method available that computes the inner approximation of a set defined by general nonlinear equality constraints \cite{jaulin01,goldsztejn06}.

Our technique is applied to power systems with a general meshed network without any modification in the system data, and the results are compared with the results obtained by MATPOWER \cite{matpower}. Our approach achieves drastic improvements in terms of conservativeness while remaining scalable to large systems. The main advantages over the existing approaches are summarized below.
\begin{enumerate}
\item The convex restriction is provided with analytical conditions using the local information and does not involve any numerical algorithm. This brings the advantage for real-time security analysis where computational capability is limited.
\item Our method is scalable to large-scale systems. We later show that the number of quadratic constraints grows proportionally to the system size. The convex restriction can be used to solve OPF replacing the non-convex power flow equations.
\item The convex restriction is guaranteed to be non-empty given the system operates in a normal condition. Moreover, the region is non-conservative and provide practical margins for operation. We visualize this region, which shows that the restriction is tight along some of the boundaries in IEEE test cases.
\end{enumerate}
}

Rest of the paper is organized as follows. In Section II, the general formulation of the problem as well as its set up in power flow equations is provided. Section III provides a guideline for constructing convex restrictions for general constraints. Section IV applies the proposed method to power flow equations and visualize the comparison between convex restrictions and true feasibility sets followed by conclusion in Section V.

\section{Convex Restriction of Feasibility Set: Formulation and Preliminaries}
\subsection{General Formulation}
Consider the following general nonlinear equality and inequality constraints with control variables $u\in\mathbf{R}^{m}$ and states variables $x\in\mathbf{R}^{n}$,
\begin{subequations}
\begin{align}
f(x,u)&=0 \label{eqn_feasibility1} \\
h(x,u)&\leq 0 \label{eqn_feasibility2}
\end{align}
\label{eqn_feasibility}
\end{subequations}
where $f:(\mathbf{R}^{n},\mathbf{R}^{m})\rightarrow \mathbf{R}^n$ and $h:(\mathbf{R}^{n},\mathbf{R}^{m})\rightarrow \mathbf{R}^r$ are vectors of functions that are continuous and differentiable. The variables are divided into control variables and internal states. Control variables are the subset of variables that can be determined freely by the system operators. State variables are the subset of decision variables that are determined by the control variables and equality constraint in \eqref{eqn_feasibility1}. 
\begin{remark}
The number of equality constraints and the number of state variables are the same. $x$ could be determined by solving the system of equations if it is solvable.
\end{remark}

Given the constraints and variables, the solvability and feasibility of control variable $u$ are defined as follows.

{\color{black}
\begin{definition}
$u$ is \textit{solvable} if there exists at least one $x$ that satisfies the equality constraint $f(x,u)=0$.
\end{definition}

\begin{definition}
$u$ is \textit{feasible} if there exists at least one $x$ that satisfies $f(x,u)=0$ and $h(x,u)\leq 0$.
\end{definition}

Feasibility sets and solvability sets refer to the set of all feasible and solvable control variables $u$. Nonlinear equality constraints create a nonlinear manifold in the space of $(x,u)$, and a singleton is the only possible convex inner approximation in a general nonlinear manifold. Instead of working with both $x$ and $u$, the feasibility set is defined as a projection of the nonlinear manifold onto the control variable space. This set is generally non-convex, and the goal of this paper is to find the convex restriction inside the projection of the nonlinear manifold. The construction of the convex restriction relies on the following assumptions.

\begin{assumption}
There is a known point $(x_0,u_0)$ that satisfies the followings:
\begin{enumerate}
    \item $f(x_0,u_0)=0,\, h(x_0,u_0)\leq0$, and
    \item $J_{f,0}=\frac{\partial f}{\partial x}\big|_{x_0,u_0}$ is non-singular.
\end{enumerate}
The known operating point $(x_0,u_0)$ will be referred as the base point.
\label{assumption1}
\end{assumption}

\begin{remark}
From Implicit Function Theorem, there exists an open neighborhood of solvability set around $u_0$ when Assumption \ref{assumption1} is satisfied.
\end{remark}
The first condition in Assumption \ref{assumption1} guarantees that the feasibility set is non-empty by enforcing the convex restriction to contain the base point.

\begin{assumption}
Nonlinear Equations have a sparse nonlinear representation. Namely there exists a basis function $\psi:(\mathbf{R}^{n},\mathbf{R}^{m})\rightarrow \mathbf{R}^q$ such that
\begin{equation}
\begin{aligned}
f(x,u)&=M\psi(x,u) \\
h(x,u)&=L\psi(x,u)
\end{aligned}
\label{eqn_psi}
\end{equation}
where $M\in\mathbf{R}^{n\times q}$ and $L\in\mathbf{R}^{r\times q}$ are constant matrices. Moreover, each $\psi_k$ is a function of small subset of $\{x_1,...,x_n\}$.
\label{assumption2}
\end{assumption}

Assumption \ref{assumption2} is necessary to ensure the scalability of the convex restriction. Implications of these assumptions in the context of power flow equations will be discussed later.}

\subsection{Power Flow Equation and Operational Constraints}
Consider a power network as a directed graph $\mathcal{G}(\mathcal{N},\mathcal{E})$ where each node in $\mathcal{N}$ represents a bus, and each edge in $\mathcal{E}\subseteq\mathcal{N}\times \mathcal{N}$ represents a transmission line. For each transmission line $l$, we will denote its \textit{from} bus with superscript $\textrm{f}$, and its \textit{to} bus as superscript $\textrm{f}$. $\mathcal{N}_\textrm{slack}$ denotes the slack bus with fixed voltage magnitude and phase angle, and $\mathcal{N}_\textrm{ns}=\mathcal{N}\backslash\mathcal{N}_\textrm{slack}$ denotes the set of non-slack buses. The set of PV buses and PQ buses are denoted by $\mathcal{N}_\textrm{pv}$ and $\mathcal{N}_\textrm{pq}$, respectively. Set of generator buses are denoted by $\mathcal{N}_G=\mathcal{N}_\textrm{pv}\cup\mathcal{N}_\textrm{slack}$. Consider the following AC power flow equations in polar coordinates with operational constraints:
\begin{equation}
\begin{aligned}
p_i^\textrm{inj}&=\sum_{k\in \mathcal{N}} v_iv_k(G_{ik}\cos\theta_{ik}+B_{ik}\sin\theta_{ik}), \ \ i\in\mathcal{N}, \\
q_i^\textrm{inj}&=\sum_{k\in \mathcal{N}} v_iv_k(G_{ik}\sin\theta_{ik}-B_{ik}\cos\theta_{ik}), \ \ i\in\mathcal{N},
\end{aligned}
\label{eqn_pf}
\end{equation}
\begin{subequations}
\begin{align}
q_i^\textrm{min}\leq&q_i^\textrm{inj}\leq q_i^\textrm{max}, \ & i&\in\mathcal{N}_G \\
v_i^\textrm{min}\leq&v_i\leq v_i^\textrm{max}, \ & i&\in\mathcal{N}_\textrm{pq} \\
\varphi^\textrm{min}_l\leq&\theta_l^\textrm{f}-\theta_l^\textrm{t}\leq\varphi^\textrm{max}_l, \ & l&\in\mathcal{E}.
\end{align}
\label{eqn_pf_feasibility}
\end{subequations}
{\color{black}where $p_i$ and $q_i$ are the active and reactive power injection, and $\theta_i$ and $v_i$ are the phase angle and voltage magnitude at bus $i$. $\theta_l^\textrm{f}-\theta_l^\textrm{t}$ denotes the phase difference between from and to end of the transmission line $l$. The operational constraints considered here are reactive power limits and voltage magnitude limits at the generators and phase angle difference limits on transmission lines. 

In the steady-state analysis of power grids, the system operator has control over the generators, which is denoted by $u$. In this paper, the feasibility of active power injection at non-slack buses will be considered so that $u=p_\textrm{ns}^\textrm{inj}$. The reactive power injection at PQ buses and voltage magnitude at PV buses are assumed to be fixed to constant values although the framework can be extended to include them. The corresponding internal states are 
$x=\begin{bmatrix} \theta_\textrm{ns}^T & v_\textrm{pq}^T \end{bmatrix}^T$. The system operators need to decide the set the control variable subject to the power flow feasibility set in equation \eqref{eqn_pf} and \eqref{eqn_pf_feasibility}. Our objective is to find a non-conservative subset around some known operating point.

The known operating point in Assumption \ref{assumption1} can be naturally chosen by the current operating point. It implies that
\begin{enumerate}
    \item the system is operating at a normal condition where the operational constraints are respected, and
    \item the system is not operating at the solvability boundary of the power flow equation.
\end{enumerate}
Assumption 2 is naturally satisfied for the power flow equations because it can be decomposed by the nonlinearity involved in transmission lines and the shunt elements. The basis functions can be chosen to be $v_iv_k\cos(\theta_{ik})$ and $v_iv_k\sin(\theta_{ik})$ for each transmission line and the voltage magnitude squares. Since the electric grid is a sparsely connected network, the number of basis functions grows approximately proportionally with respect to the system size.

\subsection{Fixed Point Representation}
The power flow equations can be converted into an equivalent fixed point form inspired by the Newton-Raphson method. Let us define the residues of basis functions around the nominal operating point as follows:
\begin{equation}
g(x,u)=\psi(x,u)-J_{\psi,0}x.
\label{eqn_f}
\end{equation}
where $J_{\psi,0}=\frac{\partial \psi}{\partial x}\big|_{x_0,u_0}$. Note that the power flow Jacobian is a linear transformation of the basis function Jacobian (i.e. $J_{f,0}=MJ_{\psi,0}$). The equality constraint can be written as
\begin{equation}
f(x,u)=J_{f,0}x+Mg(x,u),
\end{equation}
where $Mg(x)$ represents higher order terms of $f(x,u)$ after linearization. From Assumption \ref{assumption1}, the power flow Jacobian is invertable, and the equality constraint can be written in the following fixed point form:
\begin{equation}
x=-J_{f,0}^{-1}Mg(x,u).
\label{eq:fixed-point}
\end{equation}
The fixed point condition in Equation \eqref{eq:fixed-point} is an equivalent constraint to the equality condition in \eqref{eqn_feasibility1}.

\begin{remark}
The fixed point form in Equation \eqref{eq:fixed-point} is in the same form as a Newton-Raphson iteration.
\end{remark}

The Newton-Raphson method is one of the most popular algorithms for solving nonlinear equations including steady-state power flow equations \cite{mehta16}. This is widely used in practice due to its fast convergence to the solution given a good initial guess. The difference here is that the Jacobian is fixed at the base operating point while the Newton-Raphson method updates the Jacobian at every iteration. Newton-Raphson fixed point form converges quadratically in the vicinity of the solution, and it plays an important role in deriving solvability condition.}

\section{Derivation of Convex Restriction}
In this section, we describe the procedure for constructing the convex restriction for given equality and inequality constraints.
\subsection{Convex Restriction of Inequality Constraints}
First, let us consider the convex restriction of inequality constraints and ignore equality constraints. This case is much straight forward than the convex restriction with equality constraints. Suppose a vector of functions $\obar{h}(x,u)$ and $\ubar{h}(x,u)$ establishing bounds on individual components:
\begin{equation}
    \ubar{h}_k(x,u) \leq h_k(x, u) \leq \obar{h}_k(x, u).
    \label{eq:h_estimator}
\end{equation}
$\ubar{h}_k(x,u)$ and $\obar{h}_k(x, u)$ are referred as the under-estimator and over-estimator of $h_k(x, u)$, respectively. Following Lemma shows an interesting comparison between the convex restriction and convex relaxation of inequality constraints.

\begin{lemma}
Suppose under and over-estimators $\ubar{h}_k(x,u)$ and $\obar{h}_k(x, u)$ are convex functions. If $(x,u)$ is feasible for $h(x,u)\leq 0$, then
\begin{equation}
\ubar{h}(x,u)\leq 0,
\end{equation}
and the above condition forms the convex relaxation of the feasibility set.
If
\begin{equation}
\obar{h}(x,u)\leq 0,
\end{equation}
then $(x,u)$ is feasible for $h(x,u)\leq 0$, and the above condition forms the convex restriction of the feasibility set.
\label{lemma_ineq_convrestr}
\end{lemma}

\begin{figure}[!htbp]
	\centering
	\includegraphics[width=3.3in]{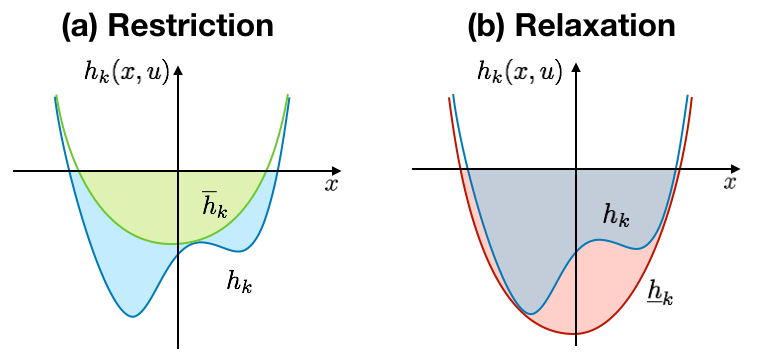}
	\caption{The relaxation and restriction of inequality constraints.}
	\label{fig_inequality_env}
\end{figure}

Lemma \ref{lemma_ineq_convrestr} shows a simple contrast between the relaxation and restriction, and Figure \ref{fig_inequality_env} graphically illustrates their differences. One observation is that the relaxation requires an envelope that encloses a convex set while the restriction requires its complementary space to enclose a convex set.

In this paper, \textit{convex envelopes} refer to the convex over-estimator and concave under-estimator, and \textit{concave envelopes} refer to the concave over-estimator and convex under-estimator. Examples of these envelopes are shown in Figure \ref{fig_convex_concave_env}. The convex envelope encloses a convex region, and it is widely used in convex relaxations of non-convex optimization problems \cite{coffrin16,tawarmalani13}. As it was shown in Lemma \ref{lemma_ineq_convrestr}, it turns out that concave envelopes are necessary for constructing convex restriction of inequality constraints. Later, we will show that even for the restriction of nonlinear equality constraints, concave envelopes need to be used to enforce convexity to the inner approximation.

\begin{figure}[!htbp]
	\centering
	\includegraphics[width=3.3in]{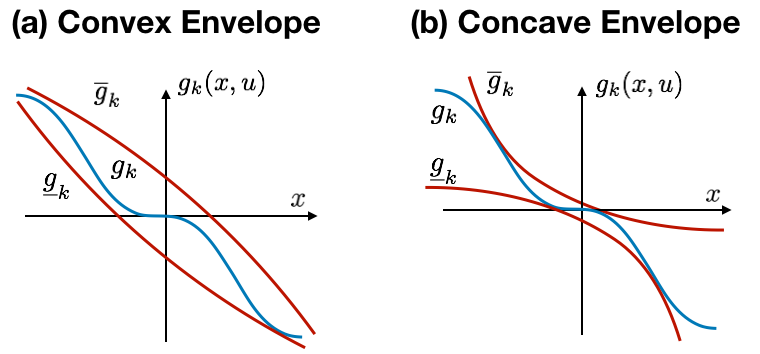}
	\caption{Examples of the convex and concave envelope.}
	\label{fig_convex_concave_env}
\end{figure}

\subsection{Preliminaries for Convex Restriction of Equality Constraints}

In this section, the convex restriction of equality constraints will be presented. The derivation will rely on Brouwer's Fixed Point Theorem, which provides a sufficient condition for the solvability of the equality constraint. Given the fixed point equation in \eqref{eq:fixed-point}, the theorem states the following.

\begin{theorem}
(\textbf{Brouwer's Fixed Point Theorem}) Let $F: \mathcal{P}\rightarrow\mathcal{P}$ be a continuous map where $\mathcal{P}$ is a compact and convex set in $\mathbf{R}^n$. Then the map has a fixed point in $\mathcal{P}$, namely $x=F(x)$ has a solution in $x\in\mathcal{P}$.
\end{theorem}

Brouwer's Fixed Point Theorem provides a sufficient condition for the existence of a solution in the internal states. The control variables $u$ can be considered as external parameters changing the fixed point equation in \eqref{eq:fixed-point}, which leads to the following Lemma.

\begin{lemma}
If $-J_{f,0}^{-1}Mg(x,u)\in\mathcal{P}$ for all $x\in\mathcal{P}$, then $u$ is solvable and has at least one solution in $x\in\mathcal{P}$.
\label{lemma_fxpt}
\end{lemma}
\begin{proof}
Let $F(x)=-J_{f,0}^{-1}Mg(x,u)$. Then, there exist a solution $x\in\mathcal{P}$ from Brouwer's Fixed Point Theorem.
\end{proof}

The existence of any self-mapping set guarantees the existence of a state solution, and the self-mapping set is not unique. This brings the idea of proposing a class of convex and compact set parametrized by some variable denoted by $b\in\mathbf{R}^p$. Instead of finding a single self-mapping set, a class of set can be used to check the Brouwer's Fixed Point condition, and the solvability region will be the union of all control variables that have a self-mapping set in the state space. The self-mapping set will be denoted by $\mathcal{P}(b)$ to show that it is parametrized by $b$. Then, the existence of $b$ such that $-J_{f,0}^{-1}Mg(\mathcal{P}(b),u)\subseteq\mathcal{P}(b)$ is sufficient for the Brouwer's Fixed Point condition. This idea can be interpreted as lifting the optimization variables to include additional variable $b$ where the construction of convex restriction is less conservative.

\subsection{Self-mapping with a Polytope Set}
While the self-mapping set can be any convex and compact set, a polytope will be considered in this paper. There is a significant computational advantage of using polytope because the set is described by inequality constrains involving just linear transformations. Let us consider a non-empty compact polytope set $\mathcal{P}$,
\begin{equation}
\mathcal{P}(b)=\{x \mid Ax\leq b\},
\label{eq_polytope}
\end{equation}
where $A\in\mathbf{R}^{p\times n}$ is a constant matrix, and $b\in\mathbf{R}^{p}$ is a vector of variables. The matrix $A$ is chosen such that it forms intervals that bounds the nonlinearity involved in the basis functions. For example, $\sin(\theta^\textrm{f}-\theta^\textrm{t})$ can be effectively bounded by choosing $A$ to be the incidence matrix. When the angle difference $\theta^\textrm{f}-\theta^\textrm{t}=E^T\theta$ has tight upper and lower bounds, then $\sin(\theta^\textrm{f}-\theta^\textrm{t})$ can be also tightly bounded. By fixing $A$ to be a constant matrix, the linear transformation does not introduce any extra complexity.

Lemma \ref{lemma_fxpt} provides a sufficient condition for the existence of internal states in $\mathcal{P}$. The condition can be extended to include inequality constraint by ensuring the self-mapping set resides inside the inequality constraints. If $h(u,x)\leq0$ for all $x\in\mathcal{P}$, the internal state solution should also satisfy $h(u,x)\leq0$. The self-mapping condition and the feasibility condition are illustrated in Figure \ref{fig_self_mapping}, and this is stated formally in the following Lemma.

\begin{figure}[!htbp]
	\centering
	\includegraphics[width=2.0in]{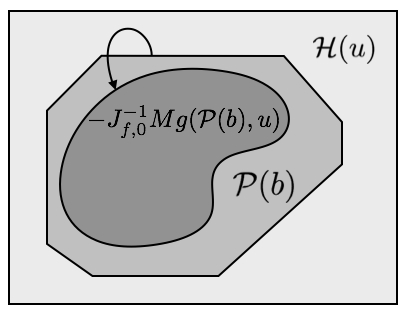}
	\caption{The self-mapping is illustrated in this figure in the domain of $\mathcal{X}$. Here, $\mathcal{H}(u)=\{x\mid h(x,u)\leq0\}$, and existence of the self-mapping set $\mathcal{P}(b)$ ensures solvability and feasibility of $u$.}
	\label{fig_self_mapping}
\end{figure}

{\color{black} 
\begin{lemma}
$u$ is feasible and there exists a corresponding state solution that satisfies $x\in\mathcal{P}(b)$ if there exists $b\in\mathbf{R}^p$ such that
\begin{equation}
\begin{aligned}
\forall & x\in \mathcal{P}(b), \ -J_{f,0}^{-1}Mg(x,u)\in\mathcal{P}(b), \\
\forall & x\in \mathcal{P}(b), \ h(x,u)\leq 0.
\end{aligned}
\label{eqn_lift}
\end{equation}

\label{lemma_lift}
\begin{proof}
The first condition ensures the self-mapping under the map $x\rightarrow -J_{f,0}^{-1}Mg(x,u)$, and thus there exists a solution $u$ for the equality constraint in equation \eqref{eqn_feasibility1} with $x\in\mathcal{P}(b)$ by Lemma \ref{lemma_fxpt}. The second condition ensures $\mathcal{P}(b)$ belongs to the feasible set for inequality constraint in equation \eqref{eqn_feasibility2}. $u$ satisfies both constraint in \eqref{eqn_feasibility1} and \eqref{eqn_feasibility2}, and thus belongs to the feasibility set.
\end{proof}
\end{lemma}

Notice that the conditions are described as an intersection of two containment conditions on the self-mapping set. The self-mapping condition for solvability can be re-written as the following condition.

\begin{lemma}
$u$ is solvable and there exists a corresponding state solution that satisfies $x\in\mathcal{P}(b)$ if there exists some $b\in\mathbf{R}^p$
\begin{equation}
\max_{x\in\mathcal{P}(b)} Kg(x,u)\leq b
\label{eqn_ADfb}
\end{equation}
where $K=-AJ_{f,0}^{-1}M$.
\label{lemma_ADfb}
\end{lemma}
\begin{proof}
The above condition is a sufficient condition to $-AJ_{f,0}^{-1}Mgf(x,u)\leq b$ for all $x\in\mathcal{P}(b)$, which shows the self-mapping of $\mathcal{P}(b)$. Then, there exists a solution $x\in\mathcal{P}(b)$ from Lemma \ref{lemma_lift}.
\end{proof}

In the next section, we find the upper-bound of the left-hand side of inequality \eqref{eqn_ADfb} by using the concave envelopes.
}

\subsection{Enclosure of Concave Envelope}
Consider over-estimator and under-estimator of $g(x,u)$, denoted by $\obar{g}(x,u)$ and $\ubar{g}(x,u)$:
\begin{equation}
    \ubar{g}_k(x,u) \leq g_k(x, u) \leq \obar{g}_k(x, u).
\end{equation}

While the above envelope gives the bound for all $x$, the inequality condition in Lemma \ref{lemma_ADfb} requires the bound over the set $\mathcal{P}$. Suppose the domain of $x$ is restricted to $\mathcal{P}$, and let us establish the bound over the self-mapping set. This bound is given by the following definition:
\begin{subequations}
\begin{align}\label{eq:max_f_over_P}
    \obar{g}_{\mathcal{P},k}(u,b) = \max_{x\in \mathcal{P}(b)} \obar{g}_k(x, u), \\
    \ubar{g}_{\mathcal{P},k}(u,b) = \min_{x\in \mathcal{P}(b)} \ubar{g}_k(x, u),
    \label{eq:min_f_over_P}
\end{align}
\label{eq:f_over_P}
\end{subequations}

Given this definition, they establish the bounds on $g_k(x,u)$ such that
\begin{equation}
    \ubar{g}_{\mathcal{P},k}(u,b) \leq g_k(x, u) \leq \obar{g}_{\mathcal{P},k}(u,b), \; \forall \, x\in\mathcal{P}(b).
\end{equation}
This forms a compact region that contains the nonlinearity in $\mathcal{P}$ as illustrated in Figure \ref{fig_envelope}. The self-mapping set could be interpreted as intervals of some transformed variables.

\begin{figure}[!htbp]
	\centering
	\includegraphics[width=2.6in]{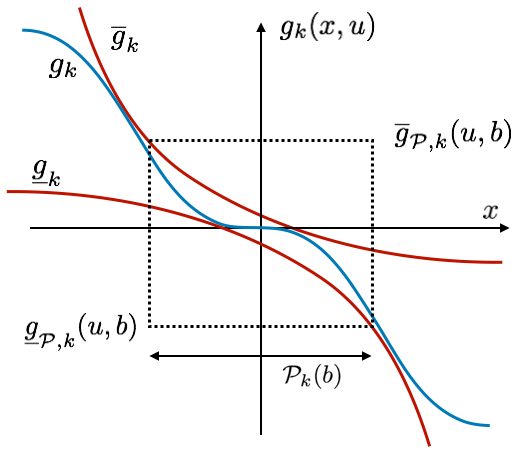}
	\caption{$\obar{g}_k(u,b)$ and $\ubar{g}_k(u,b)$ define maximum and minimum bounds of $g_k$ over $\mathcal{P}$. The dashed box contains all the nonlinearity over $\mathcal{P}$.  Notice that the upper bound always occur at the extreme points when the concave envelopes are used.}
	\label{fig_envelope}
\end{figure}

\subsection{Enforcing Convexity by Vertices Tracing}
This section introduces the vertices tracing, which is one of the key concepts that allow the scalable construction of convex restriction. Let us denote $\mathcal{P}_k$ be the polytope formed in the space of variables involved in $g_k$. Then the interval bound defined in Equation \eqref{eq:f_over_P} can be rewritten by the following Lemma.

\begin{lemma}
Suppose $\obar{g}_k(v, u)$ and $\ubar{g}_k(v, u)$ are convex and concave functions.
$\obar{g}_{\mathcal{P},k}(u,b)$ and $\ubar{g}_{\mathcal{P},k}(u,b)$ in \eqref{eq:f_over_P} are also convex and concave in $(u,b)$ and are given by
\begin{subequations}
\begin{align}
    \obar{g}_{\mathcal{P},k}(u,b) & = \max_{v\in \partial\mathcal{P}_k(b)}\obar{g}_k(v, u), \label{eqn_g_vertex1} \\
    \ubar{g}_{\mathcal{P},k}(u,b) & = \min_{v\in \partial\mathcal{P}_k(b)} \ubar{g}_k(v, u) \label{eqn_g_vertex2}
\end{align}
\label{eqn_g_vertex}
\end{subequations}
where $\partial\mathcal{P}_k(b)$ denotes the vertices of polytope $\mathcal{P}_k(b)$.
\label{lemma_gconvex}
\begin{proof}
Since $\obar{g}_k(v, u)$ is a convex function and $\mathcal{P}(b)$ is a convex function, its maximum always occur at the extreme points.
Moreover, Equation \eqref{eqn_g_vertex1} is a point-wise maximum over all vertices, therefore the convexity is preserved with respect to both $u$ and $b$ \cite{boyd04}. $\ubar{g}_k(v, u)$ can be proved in the same way.
 \end{proof}
\end{lemma}

\begin{remark}
Given Assumption \ref{assumption2}, the number of variables involved in $\psi_k$ is small. Then, the number of vertices of $\mathcal{P}_k$ is also small.
\end{remark}

Given the interval bound defined by Equation \eqref{eqn_g_vertex}, the convexity can be enforced to the self-mapping condition in Lemma \ref{lemma_ADfb}. First, positive and negative parts of matrix $K\in\mathbf{R}^{p\times q}$ are defined as $K^\pm\in\mathbf{R}^{p\times q}$ with
\begin{equation}
 K^+_{ij}=\begin{cases}K_{ij} \text{ if } K_{ij}>0 \\ 0 \text{ otherwise} \end{cases} \, K^-_{ij}=\begin{cases}K_{ij} \text{ if } K_{ij}<0 \\ 0 \text{ otherwise} \end{cases} \hskip -1em
    \label{eqn_mat_pm}
\end{equation}
where $K_{ij}$ refer to $i^{th}$ row and $j^{th}$ column of matrix $K$. So $K = K^+ + K^-$ and $\pm K^\pm_{ij} \geq 0$. The next lemma provides a convex upper-bound for left-hand side of Equation \eqref{eqn_ADfb} in Lemma \ref{lemma_ADfb}.

\begin{lemma}
{\color{black} For matrix $K\in\mathbf{R}^{p\times q}$, there exists a nonlinear map $w:(\mathbf{R}^{m},\mathbf{R}^{p})\rightarrow\mathbf{R}^{p}$ such that every entry $w_i(u,b)$ is a convex function with respect to $(u,b)$ and
\begin{equation}
 \max_{x\in\mathcal{P}(b)}K g(x, u) \leq w(u,b).
\end{equation}
$w(u,b)$ is given by
\begin{equation}
 w(u,b)= K^+ \obar{g}_\mathcal{P}(u,b) + K^- \ubar{g}_\mathcal{P}(u,b).
\end{equation}
}
\label{lemma_upper}
\end{lemma}
\begin{proof}
Since $\obar{g}_\mathcal{P}(u,b)$ and $\ubar{g}_\mathcal{P}(u,b)$ are upper and lower bounds on $g(x,u)$,
\begin{equation*}
Kg(x, u)\leq K^+ \obar{g}_\mathcal{P}(u,b) + K^- \ubar{g}_\mathcal{P}(u,b)
\end{equation*}
for all $x\in\mathcal{P}(b)$.
Moreover, $\obar{g}_\mathcal{P}(u,b)$ and $-\ubar{g}_\mathcal{P}(u,b)$ are convex and concave functions from Lemma \ref{lemma_gconvex}, and $K^+$ and $-K^-$ have non-negative entries. Therefore, convexity is preserved to $w(u,b)$ \cite{boyd04}.
\end{proof}
Lemma \ref{lemma_upper} provides us with a convex over-estimator of the self-mapping condition, and it is important to note that the function is convex. Let us first consider only the equality constraint in Equation \eqref{eqn_feasibility1}, then the following Theorem provide the convex restriction of solvability sets.
\begin{theorem}
Given a nonlinear equality constraint in Equation \eqref{eqn_feasibility1}, $u$ is solvable if there exists $b\in\mathbf{R}^p$ such that 
\begin{equation}
    K^+ \obar{g}_\mathcal{P}(u,b) +K^- \ubar{g}_\mathcal{P}(u,b)\leq b. 
    \label{eqn_convres_solvability}
\end{equation}
Moreover, $x\in\mathcal{P}(b)$.
\label{thm_solvability}
\end{theorem}
\begin{proof}
From Lemma \ref{lemma_upper}, 
\begin{equation*}
\max_{x\in\mathcal{P}(b)} Kg(x,u)\leq K^+ \obar{g}_\mathcal{P}(u,b) +K^- \ubar{g}_\mathcal{P}(u,b)\leq b.
\end{equation*}
Thus $\max_{x\in\mathcal{P}(b)} Kg(x,u)\leq b$, and $u$ is solvable with $x\in\mathcal{P}(b)$ from Lemma \ref{lemma_ADfb}.
\end{proof}

In order to incorporate inequality constraints, let us define the bound on $\psi(x,u)$ using Lemma \ref{lemma_gconvex}:
\begin{subequations}
\begin{align}
    \obar{\psi}_{\mathcal{P},k}(u,b) & = \max_{v\in \partial\mathcal{P}_k(b)}\obar{\psi}_k(v, u) \\
    \ubar{\psi}_{\mathcal{P},k}(u,b) & = \min_{v\in \partial\mathcal{P}_k(b)} \ubar{\psi}_k(v, u).
\end{align}
\label{eqn_psi_vertex}
\end{subequations}
A convex sufficient condition for $h_k(x,u)\leq0$ for all $x\in\mathcal{P}_k(b)$ can be derived using Lemma \ref{lemma_upper}. This ensures the self-mapping set is contained in the feasible set for inequality constraint (i.e. $\mathcal{P}(b)\subset\mathcal{H}(u)$). The following Theorem provides the convex restriction of feasibility set, which is the main result in this paper.

\begin{theorem}
(\textbf{Convex Restriction}) 
Given nonlinear equality and inequality constraints in Equation \eqref{eqn_feasibility1} and \eqref{eqn_feasibility2}, $u$ is feasible if there exists $b\in\mathbf{R}^p$ such that
\begin{equation}
\begin{aligned}
K^+ \obar{g}_\mathcal{P}(u,b) +K^- \ubar{g}_\mathcal{P}(u,b)&\leq b \\
 \ L^+ \obar{\psi}_\mathcal{P}(u,b) +L^- \ubar{\psi}_\mathcal{P}(u,b)&\leq0.
\end{aligned}
\label{eqn_convres_feasibility}
\end{equation}
\label{thm_feasibility}
\end{theorem}
\begin{proof}
Constraint $K^+ \obar{g}_\mathcal{P}(u,b) +K^- \ubar{g}_\mathcal{P}(u,b)\leq b$ ensures the existence of solution according to Theorem \ref{thm_solvability}. The second condition ensures that the polytope $\mathcal{P}(b)$ lies within the feasible region of inequality constraint. That is

\begin{equation*}
\max_{x\in\mathcal{P}(b)} L\psi(x,u)\leq L^+ \obar{\psi}_\mathcal{P}(u,b) +L^- \ubar{\psi}_\mathcal{P}(u,b)\leq0.
\end{equation*}

Therefore, this is a sufficient constraint for solvability of equation \eqref{eqn_feasibility1} and feasibility of equation \eqref{eqn_feasibility2}. 
\end{proof}

Note that the left-hand side of inequality \eqref{eqn_convres_feasibility} are convex functions as shown in Lemma \ref{lemma_upper}. Therefore Equation \eqref{eqn_convres_feasibility} provides convex conditions and is a sufficient convex condition for feasibility, which was the objective of the convex restriction. Moreover, the convex restriction guaranteed to be non-empty given a feasible base point stated in Assumption \ref{assumption1}.

\begin{remark}
If $\obar{g}_i(x_0,u_0)=\ubar{g}_i(x_0,u_0)$, and $\obar{\psi}_i(x_0,u_0)=\ubar{\psi}(x_0,u_0)$ (i.e. concave envelopes are tight and feasible at the base point), then the convex restriction in Equation \ref{eqn_convres_feasibility} is non-empty and contains the base point.
\end{remark}
\begin{proof}
Since $\mathcal{P}(b)=\{x\mid Ax\leq b\}$ is closed, there exists $\hat{b}$ such that $\mathcal{P}(\hat{b})=\{x_0\}$. Given the concave envelopes are tight at the base point and the base point is feasible (Assumption \ref{assumption1}), 
\begin{equation*}
\begin{aligned}
K^+ \obar{g}_\mathcal{P}(u_0,\hat{b}) +K^- \ubar{g}_\mathcal{P}(u_0,\hat{b})=Kg(x_0,u_0)&=\hat{b} \\
L^+ \obar{\psi}_\mathcal{P}(u_0,\hat{b}) +L^- \ubar{\psi}_\mathcal{P}(u_0,\hat{b})=L\psi(x_0,u_0)&\leq0.
\end{aligned}
\end{equation*}
The condition in Theorem \ref{thm_feasibility} is always satisfied at the base point, and thus the convex restriction contains the base point and is non-empty.
\end{proof}
From the above remark, a non-empty convex restriction can be always constructed around a feasible base point. The current or planned operating point can be naturally used as the base point for power flow feasibility set, which is given to the system operators through measurements. By changing the base point to different space within the feasible region, the convex restriction can be constructed at an arbitrary location.

\section{Convex Restriction of Power Flow Feasibility Set}
In this section, the convex restriction is constructed for the AC power flow equations in polar coordinates. The polar representation includes the voltage magnitude explicitly in the equation, and it is convenient to enforce the feasibility of voltage magnitude and phase limits. The AC power flow equations in equation \eqref{eqn_pf} can be written in the complex plane for all $i\in\mathcal{N}$:

\begin{equation}
p_i^\textrm{inj}+jq_i^\textrm{inj}=\sum_{k\in\mathcal{N}} Y_{ik}^Hv_iv_ke^{-j\theta_{ik}},
\end{equation}
where $Y_{ik}=G_{ik}+jB_{ik}$, and $Y_{ik}^H$ is the conjugate of $Y_{ik}$. Suppose the feasible base point has the state $\theta_0$ and $v_0$, then
\begin{equation}
p_i^\textrm{inj}+jq_i^\textrm{inj}=\sum_{k\in\mathcal{N}} \left(Y_{ik}^He^{-j\theta_{0,ik}}\right)v_iv_ke^{-j(\theta_{ik}-\theta_{0,ik})}, \ \ i\in\mathcal{N},
\end{equation}
where the base point phase is combined with the admittance matrix. Then, the phase-adjusted admittance matrix can be defined as $\widetilde{G}_{ik}+j\widetilde{B}_{ik}=Y_{ik}^He^{-j\theta_{0,ik}}$.
Let us define the difference in angle as $\varphi=E^T\theta$ and $\tilde{\varphi}=E^T\theta-E^T\theta_0$ where $E$ is the incidence matrix of the network. This can be expressed with trigonometric functions for $i\in\mathcal{N}_\textrm{ns}$ for active power and $i\in\mathcal{N}_\textrm{pv}$ for reactive power,
\begin{equation}
\begin{aligned}
p_i^\textrm{inj}&=\sum_{l\in\mathcal{E}} v_l^\textrm{f}v_l^\textrm{t}(\widehat{G}^c_{ik}\cos\tilde{\varphi}_l+\widehat{B}^s_{ik}\sin\tilde{\varphi}_l)+G_{ii}v_i^2 \\
q_i^\textrm{inj}&=\sum_{l\in\mathcal{E}} v_l^\textrm{f}v_l^\textrm{t}(\widehat{G}^c_{ik}\sin\tilde{\varphi}_l-\widehat{B}^s_{ik}\cos\tilde{\varphi}_l)-B_{ii}v_i^2,
\end{aligned}
\label{eqn_pf_eq}
\end{equation}
where $v^\textrm{f}\in\mathbf{R}^{|\mathcal{E}|}$ and $v^\textrm{t}\in\mathbf{R}^{|\mathcal{E}|}$ are voltage magnitudes at the \textit{from} and \textit{to} bus of transmission lines.  The constant matrices $\widehat{G}^c,\,\widehat{G}^s,\,\widehat{B}^c,\,\widehat{B}^s\in\mathbf{R}^{|\mathcal{N}|\times|\mathcal{E}|}$ are defined as
\begin{equation}
\widehat{G}^c_{kl}= \begin{cases} \widetilde{G}_{ik} &\mbox{if } i=l^\textrm{f} \\ \widetilde{G}_{ki} &\mbox{if } i=l^\textrm{t} \\ 0  &\mbox{otherwise} \\ \end{cases} \ 
\widehat{G}^s_{kl}= \begin{cases} \widetilde{G}_{ik} &\mbox{if } i=l^\textrm{f} \\ -\widetilde{G}_{ki} &\mbox{if } i=l^\textrm{t} \\ 0  &\mbox{otherwise} \\ \end{cases},
\end{equation}
where $l^\textrm{f}$ and $l^\textrm{t}$ are the $from$ and $to$ bus of transmission line $l$. Other matrices $\widehat{G}^s$, $\widehat{B}^c$ and $\widehat{B}^s$ are defined in the same way by simply replacing the variables.

The advantage of using Equation \eqref{eqn_pf_eq} over \eqref{eqn_pf} is that the concave envelope over trigonometric function can be systematically derived while ensuring zero gap between over and under-estimator at the base point. From the power flow equations, basis functions are chosen to be
\begin{equation}
\psi(x,u)=
\begin{bmatrix}
p^\textrm{inj} \\
q^\textrm{inj} \\
v^\textrm{f} v^\textrm{t}\cos\tilde{\varphi} \\
v^\textrm{f} v^\textrm{t}\sin\tilde{\varphi} \\
v^2
\end{bmatrix}.
\label{eq_pf_basis}
\end{equation}
With the given basis functions, the equality constraint is $f(x,u)=M\psi(x,u)=0$ where
\begin{equation}
M=
\begin{bmatrix}
I & \mathbf{0} & -\widehat{G}^c_\textrm{ns} & -\widehat{B}^s_\textrm{ns} & -G^d_\textrm{ns} \\
\mathbf{0} & I & \widehat{B}^c_\textrm{pq} & -\widehat{G}^s_\textrm{pq} & B^d_\textrm{pq}
\end{bmatrix}.
\end{equation}
$I$ and $\mathbf{0}$ are an identity matrix and a zero matrix with appropriate sizes. $G^d$ and $B^d$ are diagonal matrices with its diagonal elements equal to diagonals of $G$ and $B$, respectively. $\widehat{G}^{c}_\textrm{ns}$ denotes a matrix with only non-slack bus rows from $\widetilde{G}^c$, and $\widehat{G}^{c}_\textrm{pq}$ denotes a matrix only pq bus rows from $\widetilde{G}^c$. 
$\widehat{B}^{c}_\textrm{ns}$ and $\widehat{B}^{c}_\textrm{pq}$ are built in the same way.
Given the basis functions in \eqref{eq_pf_basis}, its residues computed using Equation \eqref{eqn_f},
\begin{equation}
g(x,u)=
\begin{bmatrix}
p^\textrm{inj}_\textrm{ns} \\
q^\textrm{inj}_\textrm{pq} \\
v^\textrm{f} v^\textrm{t}\cos\tilde{\varphi}-v^\textrm{f}_0v^\textrm{t}-v^\textrm{f}v^\textrm{t}_0 \\
v^\textrm{f} v^\textrm{t}\sin\tilde{\varphi}-v^\textrm{f}_0v^\textrm{t}_0\varphi \\
v^2-2v_0v
\end{bmatrix},
\label{eqn:g}
\end{equation}
where the omitted product is overloaded to element-wise product. For example, $v^\textrm{f}v^\textrm{t}\cos\varphi$ is element-wise product of $v^\textrm{f}$, $v^\textrm{t}$, and $\cos\varphi$. 
The self-mapping polytope is chosen as $\mathcal{P}=\{x\mid Ax\leq b\}$ where
\begin{equation}
A= \begin{bmatrix} E^T_\textrm{ns} & \mathbf{0} \\ \mathbf{0} & I \\ -E^T_\textrm{ns} & \mathbf{0} \\ \mathbf{0} & -I \\ \end{bmatrix} \ \text{and} \ \
b= \begin{bmatrix} \obar{\varphi} \\ \obar{v}_\textrm{pq} \\ -\ubar{\varphi}^\textrm{min} \\ -\ubar{v}_\textrm{pq} \end{bmatrix}
\end{equation}
and $E_\textrm{ns}$ is the incidence matrix with rows chosen for only non-slack buses.
By choosing $A$ as the above, $b$ has an interpretation of upper and lower bounds of $\varphi$ and $v_\textrm{pq}$.

The operational constraints on the voltage magnitude and phase angles can be written as $Ax\leq b^{max}$ where
\begin{equation}
b^\textrm{max}= \begin{bmatrix} {\varphi^\textrm{max}}^T & {v^\textrm{max}_\textrm{pq}}^T & {-\varphi^\textrm{min}}^T & {-v^\textrm{min}_\textrm{pq}}^T \end{bmatrix}^T.
\end{equation}

The reactive power limit constraint on PV buses can be written as $L\psi(x,u)\leq d$ where
\begin{equation}
L=\begin{bmatrix}
\mathbf{0} & \mathbf{0} & -\widehat{B}^c_\textrm{pv} & \widehat{G}^s_\textrm{pv} & -\widehat{B}^d_\textrm{pv} \\
\mathbf{0} & \mathbf{0} & \widehat{B}^c_\textrm{pv} & -\widehat{G}^s_\textrm{pv} & \widehat{B}^d_\textrm{pv}
\end{bmatrix}, \ \
d= \begin{bmatrix} q^\textrm{max}_\textrm{pq} \\ -q^\textrm{min}_\textrm{pq} \end{bmatrix}.
\end{equation}
The inequality constrained set is then $\mathcal{H}(u)=\{x\mid Ax\leq b^\textrm{max},\ L\psi(x,u)\leq d\}$. The self-mapping set belongs to the inequality constrained set ($\mathcal{P}\subseteq\mathcal{H}(u)$) if $b\leq b^{max}$ and $L^+\obar{\psi}(u,b)+L^-\ubar{\psi}(u,b)\leq d$. The trigonometric terms and its product with voltage magnitudes are bounded effectively by the phase angle differences and voltage magnitudes. In the next section, quadratic concave envelopes will be derived for bilinear and trigonometric functions, and the convex restriction will be constructed with convex quadratic constraints.

\subsection{Quadratic concave envelopes}
The main nonlinearities involved in the power flow equations in polar coordinates are the quadratic, trilinear and trigonometric functions. Following corollaries provide concave envelopes for commonly used functions that can be used as building blocks for bounding more complicated functions.

\begin{corollary}
Quadratic functions can be bounded by the following concave envelopes with the base point at $x_0$:
\begin{equation}
\begin{aligned}
x^2&\geq 2x-x_0^2 \\
x^2&\leq x^2.
\end{aligned}
\end{equation}
\label{cor_quad_env}
\end{corollary}

\begin{corollary}
Bilinear functions can be bounded by the following concave envelopes with some $\rho_1, \, \rho_2>0$ and the base point $x_0,\, y_0$:
\begin{equation}
\begin{aligned}
xy&\geq-\frac{1}{4}[\rho_1(x-x_0)-\frac{1}{\rho_1}(y-y_0)]^2 \\
& \hskip 8em +x_0y+xy_0-x_0y_0 \\
xy&\leq\frac{1}{4}[\rho_2(x-x_0)+\frac{1}{\rho_2}(y-y_0)]^2 \\
& \hskip 8em +x_0y+xy_0-x_0y_0.
\end{aligned}
\end{equation}
\label{cor_bilinear_env}
\end{corollary}
The over-estimator is tight along $\rho_2(x-x_0)-\frac{1}{\rho_2}(y-y_0)=0$, and the under-estimator is tight along $\rho_2(x-x_0)+\frac{1}{\rho_2}(y-y_0)=0$. Both over and under-estimators are tight at the base point, $(x_0,y_0)$.

\begin{corollary}
Trigonometric functions can be bounded by the following quadratic concave envelopes for all $\theta^\textrm{max} \in \left[0,\,\pi\right]$ and $\theta^\textrm{min} \in \left[-\pi,\,0\right]$:
\begin{equation}
\begin{aligned}
\sin\theta&\geq\theta+\left(\frac{\sin\theta^\textrm{max}-\theta^\textrm{max}}{(\theta^\textrm{max})^2}\right)\theta^2, \ \theta<\theta^\textrm{max} \\
\sin\theta&\leq \theta+\left(\frac{\sin\theta^\textrm{min}-\theta^\textrm{min}}{(\theta^\textrm{min})^2}\right)\theta^2, \ \theta>\theta^\textrm{min},
\end{aligned}
\end{equation}
and for all $\theta$:
\begin{equation}
\begin{aligned}
\cos\theta&\geq1-\frac{1}{2}\theta^2 \\
\cos\theta&\leq1.
\end{aligned}
\end{equation}
\label{cor_trig_env}
\end{corollary}

\begin{figure}[!htbp]
	\centering
	\includegraphics[width=3.4in]{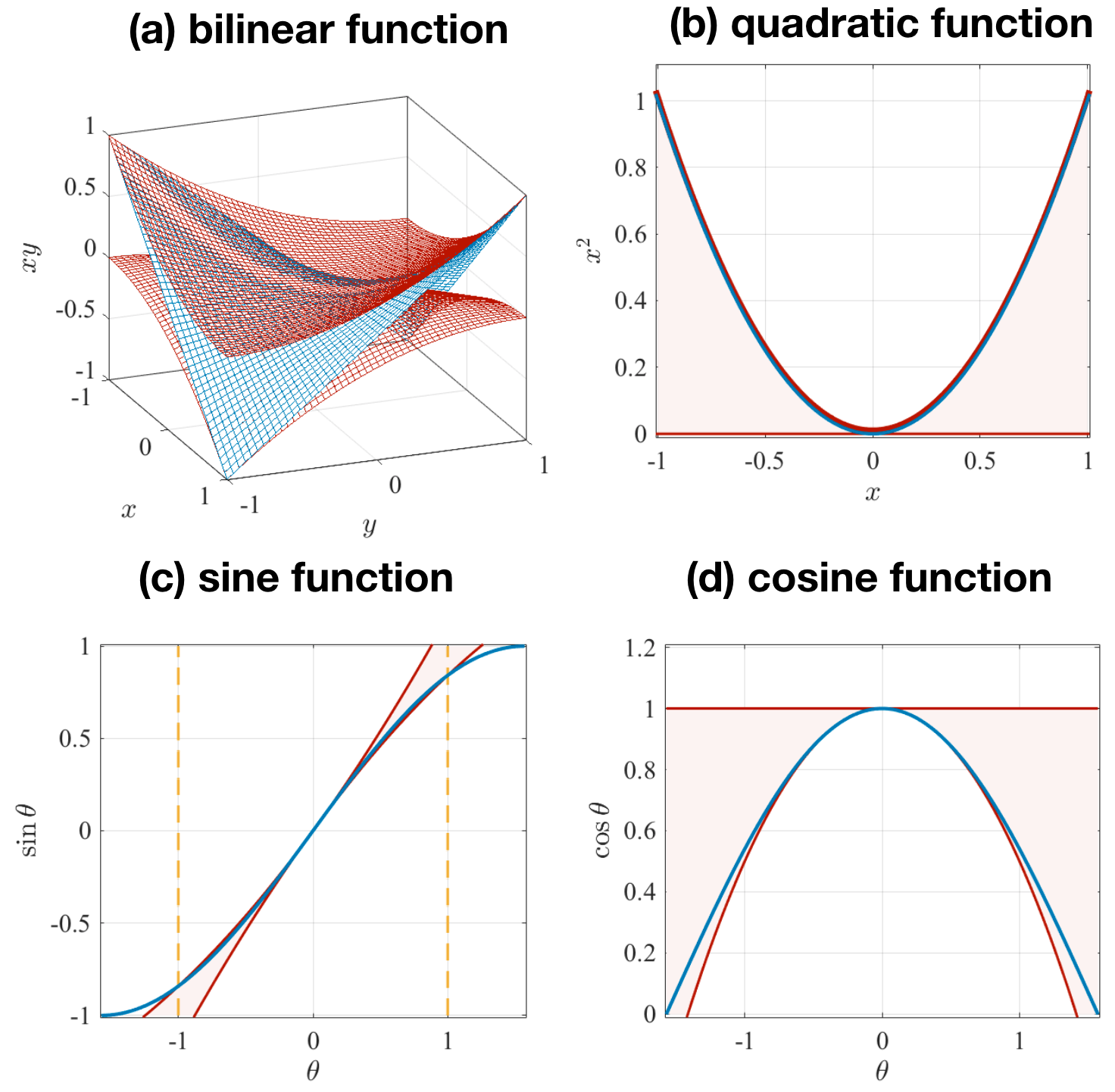}
	\caption{Illustration of concave envelopes in Corollary \ref{cor_bilinear_env} and \ref{cor_trig_env}. In (c), $\theta^{max}$ and $\theta^{min}$ is drawn with yellow dashed line.}
	\label{fig_convenv_ex}
\end{figure}

Envelopes for quadratic, bilinear and trigonometric functions are illustrated in Figure \ref{fig_convenv_ex}. More complicated functions such as trilinear functions can be bounded by cascading bilinear concave envelope. For example, $v_l^\textrm{f}v_l^\textrm{t}\cos\tilde{\varphi}_l$ can bounded by defining an intermediate variable $vv_l=v_l^\textrm{f}v_l^\textrm{t}$, and the bilinear envelope can be applied to $vv_l$ and $vv_l\cos\tilde{\varphi}_l$. In the following Lemma, we finally state the analytical expression of the convex restriction of the power flow feasibility set.
{\color{black}
\begin{corollary}
(\textbf{QC Restriction of Power Flow Equations with Operational Constraints})
The control variable $u=p_\textrm{ns}$ has at least one internal state solution, $x=\begin{bmatrix} \theta_\textrm{ns}^T & v_\textrm{pq}^T \end{bmatrix}^T$ satisfying power flow equations in \eqref{eqn_pf} and operational constraints in \eqref{eqn_pf_feasibility} if there exists $b\in\mathbf{R}^p$ such that
\begin{equation}
\begin{aligned}
K^+\obar{g}_\mathcal{P}+K^-\ubar{g}_\mathcal{P}&\leq b \\
L^+\obar{\psi}_\mathcal{P}+L^-\ubar{\psi}_\mathcal{P}&\leq d, \ b\leq b^{max} \\
\end{aligned}
\end{equation}
where
\begin{equation}
\begin{aligned}
b&=\begin{bmatrix} \obar{\varphi}^T & \obar{v}^T_\textrm{pq} & -\ubar{\varphi}^T & -\ubar{v}^T_\textrm{pq} \end{bmatrix}^T \\
\obar{g}_\mathcal{P}&=\begin{bmatrix}  p_\textrm{ns}^T & q_\textrm{pq}^T & {\obar{g}^C_\mathcal{P}}^T & {\obar{g}^S_\mathcal{P}}^T & {\obar{g}^Q_\mathcal{P}}^T  \end{bmatrix}^T \\
 \ubar{g}_\mathcal{P}&=\begin{bmatrix}  p_\textrm{ns}^T & q_\textrm{pq}^T & {\ubar{g}^C_\mathcal{P}}^T & {\ubar{g}^S_\mathcal{P}}^T & {\ubar{g}^Q_\mathcal{P}}^T  \end{bmatrix}^T \\
\obar{\psi}_\mathcal{P}&=\begin{bmatrix}  p_\textrm{ns}^T & q_\textrm{pq}^T & {\obar{\psi}^C_\mathcal{P}}^T & {\obar{\psi}^S_\mathcal{P}}^T & {\obar{\psi}^Q_\mathcal{P}}^T  \end{bmatrix}^T \\
 \ubar{\psi}_\mathcal{P}&=\begin{bmatrix}  p_\textrm{ns}^T & q_\textrm{pq}^T & {\ubar{\psi}^C_\mathcal{P}}^T & {\ubar{\psi}^S_\mathcal{P}}^T & {\ubar{\psi}^Q_\mathcal{P}}^T  \end{bmatrix}^T.
\end{aligned}
\end{equation}
These are convex quadratic constraints that provide a convex restriction of power flow feasibility set.
\label{cor_pf_convrestr}
\end{corollary}
$\ubar{g}^C_{\mathcal{P},l}$, $\ubar{g}^S_{\mathcal{P},l}$ and $\ubar{g}^Q_{\mathcal{P},l}$ denote the variables representing interval bounds of nonlinear elements in Equation \eqref{eqn:g}. Their explicit bounds are provided in the Appendix.

\begin{remark}
The number of constraints grows linearly with respect to the number of buses and number of lines. The number of constraints involved in Corollary \ref{cor_pf_convrestr} is less than $a|\mathcal{N}|+b|\mathcal{E}|$ where $|\mathcal{N}|$ and $|\mathcal{E}|$ are number of buses and transmission lines, and $a$ and $b$ are constants independent of the system size.
\end{remark}

\begin{figure*}[!htbp]
	\centering
	\includegraphics[width=7in]{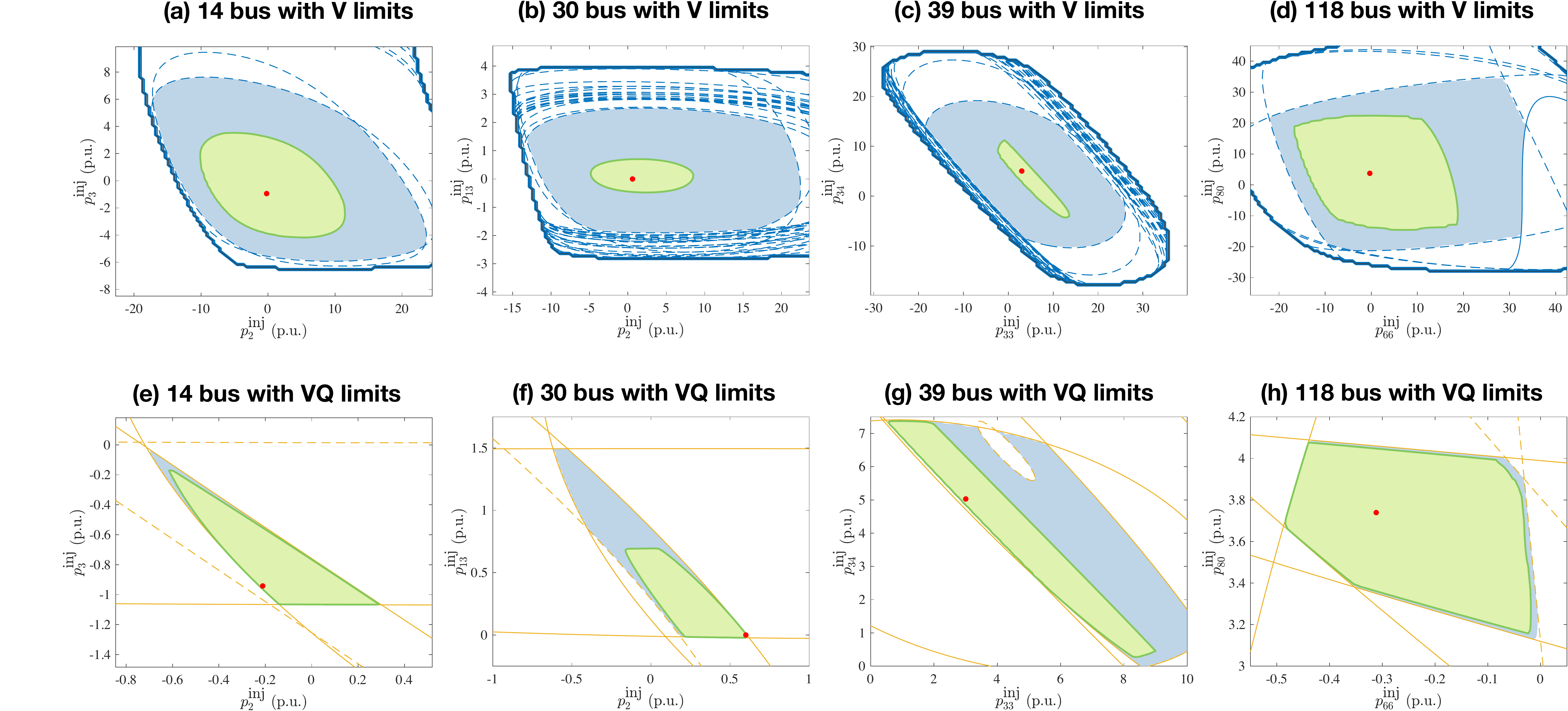}
	\caption{Convex restrictions of feasible active power injection set in 14, 30, 39 and 118 bus system are shown. Figure (a) to (d) only considers the voltage magnitude limits, and Figure (e) to (h) considers both voltage magnitude and reactive power limits. Thick blue lines show the solvability boundary. Solid yellow lines show reactive power upper limit and dashed yellow lines show reactive power lower limits.}
	\label{fig_9bus}
\end{figure*}

\subsection{Visualization of Convex Restrictions}
This section provides visualization of the convex restriction in 2-dimensional space where the constraints were implemented in JuMP/Julia \cite{jump}. The plots were drawn by varying two control variables and fixing all the others control variables. This creates a cross section of the feasibility set that cuts through the base point. The actual feasible set was solved using the Newton-Raphson method in MATPOWER package, and the same data set was used for convex restriction \cite{matpower}.

\begin{figure}[!htbp]
	\centering
	\includegraphics[width=2.8in]{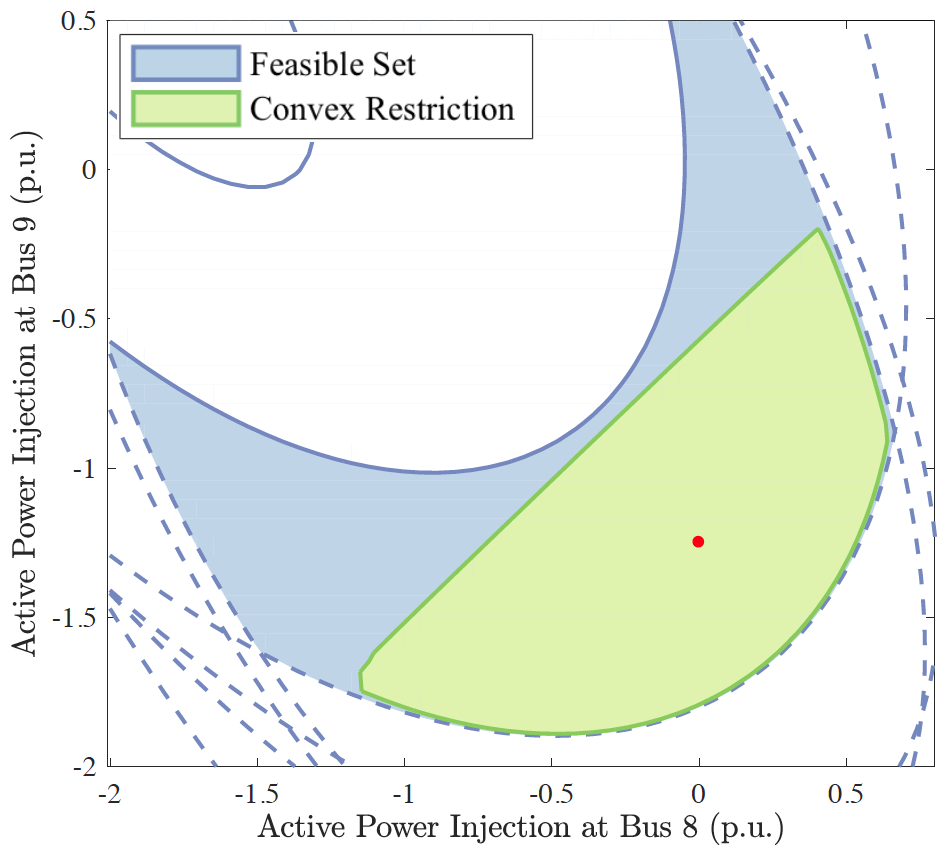}
	\caption{Convex restriction of feasible active power injection set in 9 bus system with the voltage limit of 1\% deviation from the base operating point. Red dot denotes the base point. Solid blue lines show voltage magnitude upper limits and dashed blue lines show voltage magnitude lower limits.}
	\label{fig_test9}
\end{figure}

Figure \ref{fig_test9} shows convex restriction for a modified 9 bus system. The voltage magnitude limits were set to 1\% deviation in order to create a clear non-convexity in the plot. The convex restriction was plotted by testing the feasibility of the constraint by checking violation of any operational limits. Figure \ref{fig_9bus} shows test results in a larger system for IEEE 14 bus, 30 bus, 39 bus, and 118 bus systems. The operational limits were provided in pglib library v19.01 without any modification. The results showed that the convex restriction is tight along some of the boundaries compared to the true feasibility set.
}



\section{Conclusion}
This paper proposed the convex restriction of a general feasibility set and presented its application to power flow equations with operational constraints. These results give new insights and understandings of power flow feasibility sets as a counterpart to the convex relaxation. The convex restriction of power flow feasibility set was constructed in a closed-form expression with convex quadratic constraints. The reliability of the power grid is the top priority in the operation and analysis, and the convex restriction gives a guarantee for a steady-state solution that respects operational constraints. Cross section plots of the Convex restriction in IEEE test cases showed that our construction is very close to the true feasible region along some of the boundaries. For future works, our closed-form expression can replace power flow equations to design tractable algorithms in OPF and steady-state security assessment.

\appendix
\section{Appendix}
\label{appendix}
The bounds over the self-mapping set used in the convex restriction of power flow feasibility sets are listed here. The self-mapping set forms an intersection of intervals given by $\varphi_l\in[\ubar{\varphi}_l,\, \obar{\varphi}_l]$ and $v_i\in[\ubar{v}_i,\, \obar{v}_i]$ for all $l\in\mathcal{E}$ and $i\in\mathcal{N}$. These are results directly from Lemma \ref{lemma_gconvex} with envelopes presented in Corollary 
\ref{cor_quad_env}, \ref{cor_bilinear_env} and \ref{cor_trig_env}. $\rho_1=\rho_2=1$ was used for bounding bilinear functions.

\subsection{Interval bound for cosine function}
$g_l^{\cos}=\cos\tilde{\varphi}_l-1$ over $\varphi_l\in[\ubar{\varphi}_l,\,\obar{\varphi}_l]$ is bounded by the following inequalities for all $l\in\mathcal{E}$:
\begin{equation*}
\begin{aligned}
\obar{g}_l^{\cos}&\geq 0, \ \ \ubar{g}_l^{\cos}\leq -\frac{(\varphi_{i,l}-\varphi_{0,l})^2}{2}.
\end{aligned}
\end{equation*}
where $\varphi_{i,l}\in\{\obar{\varphi}_l,\ubar{\varphi}_l\}$.

\subsection{Interval bound for sine function}
$g_l^{\sin}=\sin\tilde{\varphi}_l$ over $\varphi_l\in[\ubar{\varphi}_l,\,\obar{\varphi}_l]$ is bounded by the following inequalities for all $l\in\mathcal{E}$:
\begin{equation*}
\begin{aligned}
\obar{g}_{i,l}^{\sin}&\geq (\varphi_{i,l}-\varphi_{0,l})+\left(\frac{\sin\varphi^\textrm{min}_l-\varphi^\textrm{min}_l}{(\varphi^\textrm{min}_l)^2}\right)(\varphi_{i,l}-\varphi_{0,l})^2 \\
\ubar{g}_{i,l}^{\sin}&\leq (\varphi_{i,l}-\varphi_{0,l})+\left(\frac{\sin\ubar{\varphi}^\textrm{max}_l-\ubar{\varphi}^\textrm{max}_l}{(\ubar{\varphi}^\textrm{max}_l)^2}\right)(\varphi_{i,l}-\varphi_{0,l})^2 \\
\ubar{\varphi}_l &\leq\ubar{\varphi}_l^\textrm{max}, \ \obar{\varphi}_l\geq\obar{\varphi}_l^\textrm{min}.
\end{aligned}
\end{equation*}
where $\varphi_{i,l}\in\{\obar{\varphi}_l,\ubar{\varphi}_l\}$.

\subsection{Interval bound for bilinear function}
$g^{vv}=v_l^\textrm{f}v_l^\textrm{t}-v_{0,l}^\textrm{f}v_{0,l}^\textrm{t}$ over $v_i\in[\ubar{v}_i,\,\obar{v}_i]$ is bounded by the following inequalities for all $l\in\mathcal{E}$:
\begin{equation*}
\begin{aligned}
\obar{g}_{j,l}^{vv}&\geq \frac{1}{4}(\Delta v_{j,l}^\textrm{f}+\Delta v_{j,l}^\textrm{t})^2+v_{0,l}^\textrm{f}\Delta v_{j,l}^\textrm{t}+\Delta v_{j,l}^\textrm{f}v_{0,l}^\textrm{t} \\ 
\ubar{g}_{j,l}^{vv}&\leq -\frac{1}{4}(\Delta v_{j,l}^\textrm{f}-\Delta v_{j,l}^\textrm{t})^2+v_{0,l}^\textrm{f}\Delta v_{j,l}^\textrm{t}+\Delta v_{j,l}^\textrm{f}v_{0,l}^\textrm{t} \\
\ubar{g}_l^{vv}&\leq\ubar{g}_{j,l}^{vv}, \ \obar{g}_l^{vv}\geq\obar{g}_{j,l}^{vv}.
\end{aligned}
\end{equation*}
for each $(v_{j,l}^\textrm{f},\,v_{j,l}^\textrm{t})\in\{(\obar{v}_l^\textrm{f},\,\obar{v}_l^\textrm{t}),\, (\obar{v}_l^\textrm{f},\,\ubar{v}_l^\textrm{t}),\, (\ubar{v}_l^\textrm{f},\,\obar{v}_l^\textrm{t}),\,(\ubar{v}_l^\textrm{f},\,\ubar{v}_l^\textrm{t})\}$ and $\Delta v_l=v_l-v_{0,l}$ denotes difference respect to the base point.

\subsection{Interval bound for $v^\textrm{f}v^\textrm{t}\cos\varphi$}
$g^C=v^\textrm{f} v^\textrm{t}\cos\varphi$ and $g^C=v^\textrm{f} v^\textrm{t}\cos\varphi-v^\textrm{f}_0v^\textrm{t}-v^\textrm{f}v^\textrm{t}_0$ over $v_i\in[\ubar{v}_i,\,\obar{v}_i]$ and $\varphi_l\in[\ubar{\varphi}_l,\,\obar{\varphi}_l]$ are bounded by the following inequalities for all $l\in\mathcal{E}$:
\begin{equation*}
\begin{aligned}
\obar{\psi}_{\mathcal{P},l}^C&\geq \obar{g}_{j,l}^{vv}+v^\textrm{f}_{0,l}v^\textrm{t}_{0,l}\\ 
\ubar{\psi}_{\mathcal{P},l}^C&\leq-\frac{1}{4}(g_{j,l}^{vv}-g_l^{\cos})^2+v^\textrm{f}_{0,l}v^\textrm{t}_{0,l}g_l^{\cos}+g_{j,l}^{vv}+v^\textrm{f}_{0,l}v^\textrm{t}_{0,l}
\end{aligned}
\end{equation*}
\begin{equation*}
\begin{aligned}
\obar{g}_{\mathcal{P},l}^C&\geq \obar{g}_{j,l}^{vv}+v^\textrm{f}_{0,l}v^\textrm{t}_{0,l}-v^\textrm{f}_{0,l}v^\textrm{t}_{j,l}-v^\textrm{f}_{j,l}v^\textrm{t}_{0,l}\\ 
\ubar{g}_{\mathcal{P},l}^C&\leq-\frac{1}{4}(g_{j,l}^{vv}-g_l^{\cos})^2+v^\textrm{f}_{0,l}v^\textrm{t}_{0,l}g_l^{\cos}+g_{j,l}^{vv}+v^\textrm{f}_{0,l}v^\textrm{t}_{0,l} \\ 
&\hskip 4em -v^\textrm{f}_{0,l}v^\textrm{t}_{j,l}-v^\textrm{f}_{j,l}v^\textrm{t}_{0,l}.
\end{aligned}
\end{equation*}
for each combination of $g_l^{vv}\in\{\obar{g}_l^{vv},\,\ubar{g}_l^{vv}\}$ and $g_{i,l}^{ \sin}\in\{\obar{g}_{i,l}^{ \cos},\,\ubar{g}_{i,l}^{ \cos}\}$.

\subsection{Interval bound for $v^\textrm{f}v^\textrm{t}\sin\varphi$}

$\psi^S=v^\textrm{f} v^\textrm{t}\sin\varphi$ and $g^S=v^\textrm{f} v^\textrm{t}\sin\varphi-v^\textrm{f}_0v^\textrm{t}_0\varphi$ over $v_i\in[\ubar{v}_i,\,\obar{v}_i]$ and $\varphi_l\in[\ubar{\varphi}_l,\,\obar{\varphi}_l]$ are bounded by the following inequalities for all $l\in\mathcal{E}$:
\begin{equation*}
\begin{aligned}
\obar{\psi}_{\mathcal{P},l}^S&\geq \frac{1}{4}(g_l^{vv}+g_{i,l}^{ \sin})^2+v_{0,l}^\textrm{f}v_{0,l}^\textrm{t}g_{i,l}^{\sin} \\ 
\ubar{\psi}_{\mathcal{P},l}^S&\leq-\frac{1}{4}(g_l^{vv}-g_{i,l}^{ \sin})^2+v_{0,l}^\textrm{f}v_{0,l}^\textrm{t}g_{i,l}^{\sin}
\end{aligned}
\end{equation*}
\begin{equation*}
\begin{aligned}
\obar{g}_{\mathcal{P},l}^S&\geq \frac{1}{4}(g_l^{vv}+g_{i,l}^{ \sin})^2+v_{0,l}^\textrm{f}v_{0,l}^\textrm{t}g_{i,l}^{\sin}-v^\textrm{f}_{0,l}v^\textrm{t}_{0,l}\varphi_{i,l} \\ 
\ubar{g}_{\mathcal{P},l}^S&\leq-\frac{1}{4}(g_l^{vv}-g_{i,l}^{ \sin})^2+v_{0,l}^\textrm{f}v_{0,l}^\textrm{t}g_{i,l}^{\sin}-v^\textrm{f}_{0,l}v^\textrm{t}_{0,l}\varphi_{i,l}.
\end{aligned}
\end{equation*}
for each combination of $g_l^{vv}\in\{\obar{g}_l^{vv},\,\ubar{g}_l^{vv}\}$ and $g_{i,l}^{ \sin}\in\{\obar{g}_{i,l}^{ \sin},\,\ubar{g}_{i,l}^{ \sin}\}$.

\subsection{Interval bound for $v^2$}

$\psi^Q=v^2$ and $g^Q=v^2-2v_0v$ over $v_i\in[\ubar{v}_i,\,\obar{v}_i]$ are bounded by the following inequalities for all $k\in\mathcal{N}$:
\begin{equation*}
\begin{aligned}
\obar{\psi}_{\mathcal{P},k}^Q&\geq v^2_k \\ 
\ubar{\psi}_{\mathcal{P},k}^Q&\leq 2v_0v_k-v_0^2.
\end{aligned}
\end{equation*}
\begin{equation*}
\begin{aligned}
\obar{g}_{\mathcal{P},k}^Q&\geq v^2_k-2v_0v_k \\ 
\ubar{g}_{\mathcal{P},k}^Q&\leq -v_0^2.
\end{aligned}
\end{equation*}
where $v_k\in\{\obar{v}_k,\,\ubar{v}_k\}$.

\ifCLASSOPTIONcaptionsoff
\newpage
\fi

\bibliographystyle{IEEEtran}
\bibliography{references}
\nocite{*}

\end{document}